%% file: LINAtwoArxiv.tex
\documentclass[reqno]{amsart}
\usepackage{epsfig,latexsym,amsfonts,amssymb,amsmath,amscd,graphics,epic}
\usepackage{amsthm}
\usepackage[mathscr]{eucal}
\usepackage{mathrsfs}
\usepackage{mathbbol}
\usepackage{oldgerm,units}
 \usepackage{wrapfig}
\usepackage{ifthen}

\input tropMacro.tex

\newcommand{\vmMat}[4]{\OP \begin{matrix}
  #1 & #2 \\
  #3 & #4
\end{matrix}\CP}
\newcommand{\ds}[1]{\ {#1} \ }

\def\semiring0{semiring$^\dagger$}
\def\semirings0{semirings$^\dagger$}
\def\domain0{domain$^\dagger$}
\def\domains0{domains$^\dagger$}
\def\semifield0{semifield$^\dagger$}
\def\semifields0{semifields$^\dagger$}
\def\gperp{ {\perp \joinrel  \joinrel \joinrel  \perp }}
\def\squasilinear{strictly quasilinear}

 \def\tight{strict\  }

\def\hal{\widehat \al}
\def\scompatible{compatible}
\def\compatible{weakly compatible}
\def\nucong{\cong_\nu}

\def\nug{>_\nu}

\def\nuge{\ge_\nu}
\def\nule{\le_\nu}
\def\nule{\le_\nu}

\def\lfun{\ell}
\def\tight\ {strict}

\def\R{\Real}

\def\tT{\mathcal T}

\def\tTz{\tT_\zero}

\def\Fz{F}

\def\tlV{V^*}
\def\tlv{\widetilde v}
\def\tlw{\widetilde w}

\newcommand{\Det}[1]{ \left|{#1}\right|}

\newcommand{\Inu}[1]{\widehat{#1}}

\def\nb{\nabla}

\def\Bnb{{\Bnu{\nabla}}}

\newcommand{\Bnu}[1]{\overline{#1}}

\def\Gker{\operatorname{g-ker}}

\def\nucong{\cong_\nu}
\def\nug{>_\nu}

\def\nuge{\ge_\nu}
\def\nule{\le_\nu}
\def\nule{\le_\nu}

\def\ggsim{\, \, \curlyvee \,}
\def\hsim{ {\underset{{gd}}{\ggsim}}}

\newcommand{\etype}[1]{\renewcommand{\labelenumi}{(#1{enumi})}}
\def\eroman{\etype{\roman}}

\def\pipegs{{\underset{\operatorname{gs}}{\mid}}}

\def\lmod{\mathrel  \pipegs \joinrel\joinrel \joinrel =}

\def\lmodg{\mathrel  \pipegs   \joinrel \joinrel\joinrel =}

\def\tilG{\widetilde{G}}

\def\rank{\operatorname{rank}}

\def\nb{\nabla}
\def\ealph{\etype{\alph}}

\def\base{\tB}

\def\nb{\nabla}

\def\pSkip{\vskip 1.5mm \noindent}
\def\tmap{\varphi}
\def\invr{{\operatorname{-1}}}

\def\rad{\operatorname{rad}}

\def\a{\alpha}
\newtheorem{thm}[theorem]{Theorem}
\newtheorem*{thm*}{Theorem}

\def\Homgs{\operatorname{Hom}_{\operatorname {gs}}}
\newtheorem{lem}[theorem]{Lemma}
\newtheorem{rem}[theorem]{Remark}
\newtheorem{prop*}{Proposition}

\newtheorem{prop}[theorem]{Proposition}
\newtheorem{defn}[theorem]{Definition}
\newtheorem*{examp*}{Example}
\newtheorem*{examples*}{Examples}
\newtheorem*{remark*}{Remark}
\newtheorem*{defn*}{Definition}

\def\lfun{\ell}
\def\R{\Real}

\def\tT{\mathcal T}

\def\tTz{\tT_\zero}

\numberwithin{equation}{section}

\def\M0{M_{\zero}}

\def\SR{R}

\def\tGz{\mathcal G_\zero}
\def\tHz{\mathcal H_\zero}

\def\Rd{R^\dagger}
\def\Fd{F^\dagger}
\def\nb{\nabla}

\def\Bnb{{\Bnu{\nabla}}}

\def\oA{A_{\tB}}
\def\trn{{\operatorname{t}}}
\def\tilL{\widetilde{L}}

\def\otB{\overline{\tB}}

\def\PS{P}

\def\rone{{\one_\SR}}

\def\fone{\one_F}
\def\fzero{\zero_F}

\newcommand{\nPS}[1]{\PS_{(!#1)}}
\newcommand{\nPSo}[1]{\nPS{\one}}

\newcommand{\bil}[2]{\langle{#1},{#2}\rangle}

\newcommand{\adj}[1]{\operatorname{adj}({#1})}

\begin{document}


\title[Dual spaces and bilinear forms]
{Dual spaces and bilinear forms \\[1mm] in supertropical linear algebra}

\author[Z. Izhakian]{Zur Izhakian}
\address{Department of Mathematics, Bar-Ilan University, Ramat-Gan 52900,
Israel} \email{zzur@math.biu.ac.il}

\author[M. Knebusch]{Manfred Knebusch}
\address{Department of Mathematics, University of Regensburg, Regensburg,
Germany} \email{manfred.knebusch@mathematik.uni-regensburg.de}

\author[L. Rowen]{Louis Rowen}
\address{Department of Mathematics, Bar-Ilan University, Ramat-Gan 52900,
Israel} \email{rowen@math.biu.ac.il}

\thanks{The work of the first and third authors has been supported by the Israel Science Foundation, grant 448/09.}

\thanks{The second author was supported in part by the Gelbart Institute at
Bar-Ilan University, the Minerva Foundation at Tel-Aviv
University, the Mathematics Dept. of Bar-Ilan University, and the
Emmy Noether Institute}


\date{\today}




\begin{abstract}

Continuing \cite{IzhakianKnebuschRowen2010LinearAlg}, this paper
 investigates finer points of supertropical vector spaces, including dual bases and  bilinear
forms,
 with supertropical versions of standard classical results such as the
 Gram-Schmidt theorem and Cauchy-Schwartz inequality,
 and change of base. We also present the supertropical version of
 quadratic forms, and see how they correspond to symmetric supertropical
 bilinear forms.
\end{abstract}

\maketitle




\section{Introduction}
\numberwithin{equation}{section}

This paper, the continuation of
\cite{IzhakianKnebuschRowen2010LinearAlg}, brings  the analog of
some classical   theorems of linear algebra to the supertropical
setting. The major difference of supertropical linear algebra from
classical linear algebra is that one can have proper subspaces of
the same rank, which we call \textbf{thick}.

 We also consider linear
maps in the supertropical context, for which the equality $\tmap
(v+w) = \tmap (v) + \tmap(w)$ is replaced by the ghost surpassing
relation $\tmap (v+w) \lmod \tmap (v) + \tmap(w)$. Supertropical
linear maps lead us to the notion of the supertropical
\textbf{dual space}. The dual space depends on the choice of thick
subspace with s-base $\tB$, but there is a natural ``dual s-base''
of $\tB$, of the same rank (Theorem~\ref{dualbasethm}). This leads
to rather delicate considerations concerning dual spaces,
including an identification of a space with its double dual in
Theorem~\ref{dualiso}.

 To understand angles, we study  supertropical bilinear
forms. As usual, in the supertropical theory the zero element  is
replaced by the ``ghost ideal.'' This complicates our approach to
bilinear forms, since the theory can be distorted by the inner
product of two elements being a ``large'' ghost. Thus, we
introduce a condition (Definition~\ref{compat}) to control the
$\nu$-value of the inner product, to prevent it from obscuring
tangible angles, which follows from an analog of the Cauchy-Schwartz inequality (cf.~Definition~\ref{CauS}). 
 Then we  also get a supertropical analog of the
Gram-Schmidt process in Lemma~\ref{GS0} and~
Theorem~\ref{decompthm}.

As with the classical theory, one can pass back and forth from
bilinear forms to quadratic forms. Surprisingly, at times  this is
easier in the supertropical theory, as seen in
Theorem~\ref{quadlin}, because many supertropical quadratic forms
satisfy the quasilinear property of Definition~\ref{quasil}.

\subsection{Background}
%

Let us briefly reviewing briefly the supertropical foundations. A
\textbf{semiring without zero}, which we notate as \semiring0, is
 a structure $(\Rd ,+,\cdot, \rone)$ such that $(\Rd ,\cdot \,
,\rone)$ is a monoid and $(\Rd ,+)$ is a commutative semigroup,
with distributivity of multiplication over addition on both sides.
 A \textbf{supertropical \semiring0}  is a
triple $(\Rd,\tG, \nu),$ where $\Rd$ is a \semiring0    and $ \tG
$ is a \semiring0 ideal, called the \textbf{ghost ideal}, together
with an idempotent map
$$\nu : \Rd \To \tG$$ (preserving multiplication as well as
addition) called the \textbf{ghost map on} $\Rd$, satisfying the
following properties, where we write $a^\nu$ for $\nu(a)$:
\begin{enumerate} \ealph
 \item  $a+b   =  a^{\nu } \quad \text{if}\quad a^{\nu } =
 b^{\nu}$; \pSkip
 \item $a+b  \in \{a,b\},\ \forall a,b \in \Rd \ s.t. \  a^{\nu }
\ne b^{\nu }.$

 (Equivalently, $\tG$ is ordered, via $a^\nu \le
b^\nu$ iff $a^\nu + b^\nu = b^\nu$.) \pSkip
\end{enumerate}

  In particular,  $a^\nu = a+a.$ We write $a \nug b$ if $a^{\nu } >
 b^{\nu}$; we stipulate that $a$ and
$b$ are $\nu$-{\bf matched}, written $a \nucong b$,  if $a^{\nu }
=
 b^{\nu}$. We say that  $a$  {\bf dominates} $b$ if  $a \nug  b$;  $a$  {\bf weakly dominates} $b$ if  $a \nuge  b$.

Recall that any commutative supertropical semiring satisfies the
\textbf{Frobenius formula} from
\cite[Remark~1.1]{IzhakianRowen2007SuperTropical}:
\begin{equation}\label{eq:Frobenius} (a+b)^m = a^m + b^m
\end{equation}
for any
    $m \in \Net^+$.

A \textbf{supertropical \semifield0}  is
 a supertropical~\semiring0 $\Fd$ for which  $$\tT : = \Fd \setminus \tG$$  is  a group, such that the map $\nu
|_\tT : \tT \to \tG$ (defined as the restriction from $\nu$ to
$\tT$) is onto. $\tT$ is  called the set of \textbf{tangible
elements} of $\Fd$. Thus, $\tG$ is also a group. 

 A \textbf{supertropical vector space} over a supertropical \semifield0 $\Fd$ is just a \semiring0 module (satisfying the usual
 module axioms,
 cf.~\cite[Definition~2.8]{IzhakianKnebuschRowen2010LinearAlg}).   $V$ has the
distinguished \textbf{standard ghost subspace} $\tHz : = e V,$ as
well as the \textbf{ghost map} $\nu: V \to \tHz,$ given by $\nu(v)
:= v+v = e v$. We write $v^\nu$ for $\nu(v)$.

When dealing with vector spaces,  we will assume for convenience
of notation   the existence of a zero element  $\fzero \in F$.
More precisely, one could start with a \semifield0 $\Fd$ and then
consider the formal vector space $F := \Fd \cup \{ \fzero \}$. A
nonzero vector $v$ of $F^{(n)}$ is called \textbf{tangible} if
each of its components is in $\tT \cup \{ \fzero \}.$

\begin{definition} We define the \textbf{ghost
surpasses} relation $\lmodg$ on any
 supertropical \semifield0 $\Fd$ (resp.~on a supertropical vector space $V$), by $$b \lmodg a \qquad \text{ iff } \qquad b = a +
c \quad \text{for some   ghost element} \quad  c,$$ where $a,b,c
\in \Fd$ (resp.~$a,b,c \in V$).
\end{definition}
\noindent This relation is antisymmetric, by \cite[Lemma
1.5]{IzhakianRowen2009Equations}.
 In this notation, by writing $a \lmodg \fzero$ we mean
$a \in \tHz$.

%
%
%
%
\subsection{Matrices}
 Assume that
$A$ is a nonsingular matrix. We define the matrices $$A^\nb  =
\frac \fone {\Det{A}} \adj{A} , \qquad  A^\Bnb := A^\nb A A^\nb,$$
cf. \cite[Remark 2.14]{IzhakianRowen2009Equations}, and recall
that $I_A = A A^\nb$
and $I'_A = A^\nb A$ 
are quasi-identities, in the sense that they are multiplicatively
idempotent matrices having determinant $\fone,$ and ghost surpass
the identity matrix.

   Then the matrices    $$I_A,
\quad I'_A , \quad A^\Bnb = A^\nb A A^\nb = A^\nb I_A, \quad
\text{and } \ I_A A$$ are nonsingular,  since $I_A A A^\nb = I_A^2
= I_A$ is nonsingular.

\subsection{Bases}
In \cite{IzhakianRowen2008Matrices}, we  defined vectors in $V$ to
be \textbf{tropically independent} if no  linear combination with
tangible coefficients is in $\tHz$, and proved that a set of $n$
vectors is tropically independent iff its matrix has rank~$n$.


\begin{defn}\label{def:thick}  A
\textbf{d-base} (for {\it dependence base}) of a supertropical
vector space $V$ is a maximal  set of tropically independent
elements of $V$. Although d-bases could have different number of
elements, we define $\rank(V)$ to be the maximal possible
cardinality of a d-base.

A subspace $W$ of a supertropical vector space  $V$ is
\textbf{thick} if $\rank(W)= \rank(V).$

 An \textbf{s-base}
of $V$ (when it exists) is a minimal spanning set.

A \textbf{d,s-base} is a d-base which is an s-base.
A vector $v \notin \tHz$ in $V$ is \textbf{critical} if we cannot
write $v \lmodg v_1 + v _2$ for $v_1,v_2 \in V \setminus Fv.$
\end{defn}

In \cite[Theorem 5.24]{IzhakianKnebuschRowen2010LinearAlg} we
prove that  the s-base (if it exists) is
 unique up to multiplication by scalars.

\begin{example}\label{sb} The \textbf{standard d,s-base} for $\Fz  ^{(n)}$ is the
set of
 vectors $$\{ (\fone, \fzero, \fzero\dots, \fzero), (\fzero, \fone,
 \fzero, \dots, \fzero), \dots, (\fzero, \dots, \fzero, \fone)\}.$$
\end{example}

Given the plethora of thick subspaces, one would expect the theory
of dual spaces to be rather complicated, and one of our basic aims
in this paper is to make sense of duality.

 Bilinear forms are introduced in \cite[Section~6]{IzhakianKnebuschRowen2010LinearAlg}
 in order to treat orthogonality of vectors. We review them
  in this paper as they are needed, in \S\ref{bilform}.

\section{Supertropical linear maps   and the dual space}

In this section we introduce supertropical linear maps, and use
these to define the dual space with respect to a d,s-base $\tB$,
showing that it has the canonical
dual s-base to be given in Theorem~\ref{dualbasethm}. 
(A version of a dual space for idempotent semimodules, in the
sense of dual pairs, given in \cite{CGQ}, leads to a Hahn-Banach
type-theorem.)

\subsection{Supertropical maps}

Recall that a \textbf{linear map} $\varphi:V\to V' $ of vector
spaces over a semifield~$F$ satisfies
$$\varphi(v+w) = \varphi(v) + \varphi(w),\qquad \varphi(av) = a\varphi(v), \qquad \forall a \in R, \ v,w \in V.$$

We weaken this a bit in the supertropical theory.

\begin{defn}\label{tropmp} Given supertropical vector spaces $ V $ and
$ V' $ over a supertropical semifield~$F$, a \textbf{supertropical
map}
$$\tmap : V     \ \to \  V'$$ is a function satisfying
\begin{equation}\label{tamp0}\tmap (v+w)
\lmod \tmap (v) + \tmap (w),\qquad \tmap (\a v) = \a\tmap (v),
\qquad \forall \a \in \tT, \ v,w \in V.\end{equation}

We write $ \Homgs (V, V' )$ for the set of supertropical maps from
$V$ to $V',$ which is viewed as a vector space over $F$ in the
usual way, given by pointwise operations. We write $\tHz : = eV$
and $\tHz' : = eV'.$
 \end{defn}

\begin{lem}\label{fctl1} Any supertropical map $\tmap :V
\to V'$ satisfies
$$\tmap(a v)
\lmodg a \tmap(v)$$ for any $v\in V$ and $a \in F $. In
particular,
\begin{equation}\label{tamp1} \tmap(\tHz) \subseteq
\tHz'.\end{equation}
\end{lem}
\begin{proof} The assertion holds by definition when $a\in \tT,$ and when $a\in \tG$
we take $\a \in T$ such that  $a = \a^ \nu = \a+\a$ and thus have
$$\tmap(a v) = \tmap((\a+\a) v) \lmodg \tmap(\a v)+\tmap(\a v) = \a\tmap( v)+\a\tmap( v) = (\a+\a)\tmap( v) = a\tmap( v).$$
The last assertion follows by taking $a = e.$\end{proof}

\begin{rem}\label{fctl2} One may wonder why we have required  $\tmap (\a v) = \a\tmap
(v)$ and not just $\tmap (\a v) \lmod \a \tmap (v)$. In fact,
these are equivalent when $\a \in \tT$, since $F$ is a
supertropical semifield. Indeed, assume that $\tmap (\a v) \lmod
\a \tmap (v)$ for any $\a  \in \tT$ and $v \in V$. Then also $\a
^{-1} \in \tT$. By hypothesis,
$$\a ^{-1}\tmap (\a v) \lmod \a ^{-1} \a \tmap (v)= \tmap(v)$$ and
$$\tmap (v) = \tmap (\a ^{-1} \a  v) \lmod \a ^{-1}\tmap (\a  v),$$ so by
antisymmetry, $\a ^{-1}\tmap (\a v) = \tmap(v),$ implying $\tmap
(\a v) = \a \tmap(v).$\end{rem}

\begin{rem} Lemma~\ref{fctl1} implies
$$\tmap(v^{\nu}) = \tmap(e v) \lmod e \tmap(v) = \tmap(v)^{\nu
};$$ i.e., $ \tmap  \circ \nu \lmod \nu \circ \tmap.$
\end{rem}

\begin{lem}\label{gsur} If $v \lmodg w$ then $\tmap(v)\lmodg
\tmap(w).$   \end{lem}
\begin{proof} Write $v = w + w'$ where $w' \in \tHz.$ Then  $$\tmap(v) \lmodg \tmap(w) + \tmap(w') \lmodg \tmap(w)$$
since $\tmap(w) \in \tHz'$. \end{proof}

\begin{rem} $ \Homgs (V, V' )$ has a supertropical vector space
structure, under the natural operations
$$(\varphi_1 + \varphi_2)(v) = \varphi_1(v) + \varphi_2(v), \qquad (a\varphi)(v) = a\varphi(v),
 \qquad \nu(\varphi)(v) = \varphi(v)^\nu,$$
for $\varphi \in \Homgs (V, V' )$, $ a \in F$, $v \in V$.

The \textbf{ghost maps} are   $\{ f\in \Homgs (V, V' ): f(V)
\subseteq \tHz'\}$.\end{rem}

\begin{defn}\label{def:ghostKer}
Given a supertropical map   $\tmap : V \to V'$, we define the
\textbf{ghost  kernel} $$\Gker (\tmap) := \tmap^{\invr}(\tHz') =
\{ v\in V: \tmap(v) \in \tHz'\}.$$  We say that $\tmap$ is
\textbf{ghost monic} if $\tmap^{\invr}(\tHz') = \tHz.$
\end{defn}

\begin{rem} $\Gker (\tmap)$ is an $F$-subspace   of $V$.\end{rem}

\begin{defn}\label{def:onto}
A  supertropical  map $\tmap  : V \to W$ of vector spaces  is
called
 \textbf{tropically onto} if $\tmap(V)$ contains a thick subspace of $W$.
 An \textbf{iso} is a supertropical map that is both ghost monic and
  tropically onto. (Note that this need not be an isomorphism in the
  usual sense, since $\tmap$ need not be onto.)
\end{defn}

\begin{rem} The composition of isos is an iso.
\end{rem}



\subsection{Linear functionals}
\begin{defn} Suppose $V$ is a vector space
over a supertropical semifield $F$. The space of supertropical
maps
$$\tlV := \Homgs (V, \Fz ),$$
  is called the
\textbf{(supertropical) dual $F$-space} of $V$, and  its elements
are called \textbf{linear functionals};   i.e., any linear
functional $\lfun \in \tlV$ satisfies
$$\lfun(v_1 + v_2)  \lmod \lfun(v_1)  + \lfun(v_2) , \qquad  \lfun(\a v_1)
= \a \lfun(v_1)$$ for any $v_1, v_2 \in V$ and $\a \in \tT $.

The set $ \tHz (V^*)$ of \textbf{ghost linear functionals} is the
set of linear functionals that are ghost maps, i.e., $\{ \lfun \in
\tlV: \lfun (V) \subseteq \tGz\}$.
\end{defn}
%

%

 Our next goal is to describe the linear functionals for
thick subspaces $V$ of $\Fz  ^{(n)}$
(including the case $V=\Fz ^{(n)}$). Towards this end, we want a
definition of linear functionals that respects a given d-base $\tB
= \{b_1, \dots, b_n\}$  of~$V$. We define the \textbf{matrix}
$A(\tB)$ of $\tB$, to be the matrix whose columns are the vectors
comprising $\tB$. For the remainder of this section we set the
matrix
$$ A : = A(\tB) .$$

\begin{defn} A d-base $\tB$ is \textbf{closed} if $I_A \tB =
\tB.$ A \textbf{closed subspace}  of $V$  is a subspace having a
closed d-base.
\end{defn}

There is an easy way to get a closed d-base from an arbitrary
d-base $\tB.$

\begin{defn}\label{sdbase0} Write $\oA = I_A A$,  and let
$\otB$ denote the set of column vectors
 of $\oA.$
 Let $$V_\tB : = \{ \oA v : v \in V
\},$$   the  subspace of $V$ spanned  by $ \otB.$ \end{defn}

\begin{rem} $\oA = I_A A$ is a nonsingular matrix, implying $\otB$  is a  d-base. $\otB$ is
easier to compute  than~ $\tB,$ since now we have
 $$\oA A^\nb I_A = AA^\nb AA^\nb I_A = I_A^3 = I_A,$$ implying
$$I_{A} \oA =  I_A  I_A A =  I_A A = \oA.$$\end{rem}

\begin{lem}$V_\tB$ is a thick, closed subspace of $V$, and
  $\tB$ is a closed d,s-base of $V_\tB.$
\end{lem} \begin{proof} $V_\tB$ contains $n$ independent vectors. Clearly $\otB$ is
closed since $I_A^2 =I_A$.
 \end{proof}

 Rather than
dualizing all of $V$, we turn to the space
$$\tlV_{\tB} : = \Homgs (V_\tB, \Fz ).$$

 Define $L_A \in \Homgs (V,V)$ by
$$L_A(v) := A^\Bnb v, \qquad \text{for every } v \in V.$$
We also define the map $\tilL_A: V \to V $ by
$$\tilL_A(v) :=  I_A v, \qquad \text{for every } v \in V.$$

\begin{rem}
$(\tilL_A)^2 = \tilL_A$, and $ \tilL_A$ is the identity on
$V_{\tB}$ since $$I_A (I_A Av) = I_A ^2 A v = I_A  A v . $$

Likewise, $L_A(v) = v$ for all $v \in V_\tB.$
\end{rem}

\begin{lem}If $\lfun \in \tlV_{\tB}$, then  $\lfun = (\lfun \circ
\tilL_A )\big|_{V_{\tB}}$ on $V_{\tB}$. In other words,
$$\tlV_{\tB} = \{ {(\ell \circ \tilL_A)} \big|_{\tB} :
 \ell \in \tlV\}.$$  \end{lem}
\begin{proof}  Follows at once from the remark.\end{proof}

\begin{lem}
$\tHz(\tlV_{\tB})=\{  f |_ {V_{\tB}} : f \in  \tHz (V^*)\}$.
\end{lem}
\begin{proof} Suppose $f' \in \tHz(\tlV_{\tB}).$ Let $f = f'\circ \tilL_A  \in \tHz (V^*).$ Then
$f' = f |_ {V_{\tB}}.$ The other inclusion is obvious.
\end{proof}

 \begin{defn}
Given a  closed  d-base $\tB= \{ b_1, \dots, b_n\}$ of $V,$ define
 $ \ep_i: V_{\tB} \to \Fz $ by
 $$ \ep_i(v)    = {b_i} ^\trn  L_A( v)=b_i^\trn A^\Bnb v,$$
the scalar product of $b_i$ and $A^{\Bnb}v.$ Also, define $\tB^* =
\{\ep_i :  i =1, \dots,  n\} $. \end{defn}

When $v$ is tangible, we saw in \cite[
Remark~4.19]{IzhakianKnebuschRowen2010LinearAlg} that $$v \ \hsim
\ \sum _{i=1}^n  \Inu{\ep_i(v)}b_i $$ is a saturated tropical
dependence relation of $v$ on the $b_i$'s; this is the motivation
behind our definition.

\begin{rem}\label{rmk:linearFun}

  $ \ep_i$ is a linear functional. Also, by definition,
  $\ep_i(b_j)$ is the $(i,j)$ position of $A A^\nb = I_A,$ a
  quasi-identity, which implies
  $$\ep_i(b_i) = \fone; \qquad \ep_i(b_j)\in \tGz ,\ \forall i\ne j.$$
  Hence, $$\sum _{i=1}^n \a _i \ep_i (b_j) \lmodg \a _j \ep_j (b_j)
  = \a _j.$$

%
  \end{rem}

\begin{thm}\label{dualbasethm} If $F$ is a supertropical semifield and $\tB$ is a closed d-base of $V$,
then $\{\ep_i : i = 1, \dots,  n\}$ is a closed d,s-base of
$\tlV_{\tB}$.
\end{thm}
\begin{proof} For any $\lfun \in V_\tB^*$, we write $\a _ i = \lfun(b_i) $,   and then see from Remark
\ref{rmk:linearFun} that $\sum_{i = 1}  ^{n} \a_i \ep_i  \lmodg
\lfun $ on $V_{\tB}$.

It remains to show that the $\{\ep_i : i = 1, \dots, n\}$ are
tropically independent.
 If $\sum_{i=1}^n \bt_i \ep_i$ were ghost for some $\bt_i \in \tTz$,
  we would have $\sum_{i = 1}  ^{n} \bt_i  {b_i}^\trn A^\Bnb$ ghost.
Let $D$ denote the diagonal matrix  $\{\bt_1, \dots, \bt_n\},$
  and let  $\mathcal I = \{ i: \bt_i \ne \fzero \},$ and assume
  there are $k$ such tangible coefficients $ \bt_i$. Then  for any
  $i\notin \mathcal I$ we have $\bt_i = \fzero $, implying the $i$ row
 of the matrix  $D I_A$ is zero. But
  the sum of the rows of the matrix~$D I_A$ corresponding to indices from
  $\mathcal I$
 would be $\sum_{i = 1}  ^{n} \bt_i  {b_i}^\trn   A^\Bnb$, which is ghost, implying that these $k$ rows of  $D I_A$
are dependent; hence  $D I_A$ has rank $\le k-1$. On the other
hand, the $k$ rows of  $D I_A$ corresponding to indices from
$\mathcal I$ yield a $k \times k$ submatrix of determinant $\prod
_{i \in \mathcal I} \bt_i \in \tT,$ implying its rank $\ge k$ by
\cite[Theorem 3.4]{IzhakianRowen2009TropicalRank}, a
contradiction.
   \end{proof}

  In the view of the theorem, we denote $\tB^* = \{\ep_i :  i =1,\dots ,  n\}$, and call it the
  (tropical)
\textbf{dual d,s-base} of $V_{\tB}$.
%
%
%
%
%
 Write $V^{**}_\tB$ for $(\tlV_\tB)^*$. Define a map $$\Phi: V_\tB \to
V^{**}_\tB,$$  given by  $v \mapsto f_v$, where
$$f_v(\lfun) \ = \ \lfun(v).$$

\begin{example}\label{6.23} 
The map $\Phi: \Fz  ^{(n)}\to {\Fz  ^{(n)}}^{**}$ is a vector
space isomorphism when $\base$ is the standard base
(cf.~Example~\ref{sb}).
\end{example}

\begin{rem} 
Since $A A^\nb = I_A$ is a quasi-identity matrix, we see that
$$f_{b_j}(\ep_i) = \ep_i(b_j) = {b_i}^\trn A^\Bnb \, b_j$$
\end{rem}

\begin{thm}\label{dualiso} Suppose $V $ is a thick closed subspace of
$F^{(n)},$ with a d,s-base of tangible vectors. For any $v\in V,$
define $v^{**}\in V^{**}$ by $v^{**}(\ell) = \ell(v).$ The map
$\Phi: V\to V^{**} $ given by $v \mapsto v^{**}$ is an iso of
supertropical vector spaces.
\end{thm}

\begin{proof} 
Applying Theorem~\ref{dualbasethm} twice, we see that $\Phi(\tB)$ is a d-base of $n$
elements.
 $\Phi$ is ghost monic, since
$\Gker \Phi$ cannot contain tangible vectors, in view of
\cite[Theorem 3.4]{IzhakianKnebuschRowen2010LinearAlg} (which says
that the g-annihilator of a nonsingular matrix cannot be
tangible).
%
But by Example~\ref{6.23}, taking the standard classical base, we
see that $V^{**}$ has  rank $n$, and thus is thick in $F^{(n)}$.
\end{proof}

 \section{Supertropical bilinear forms}\label{bilform}

Linear functionals and dual spaces   cast more light on the
supertropical theory of bilinear forms. Here is a more concise
version of
\cite[Definition~6.1]{IzhakianKnebuschRowen2010LinearAlg}.
Throughout, $F$ denotes a supertropical \semifield0 (although we
permit the possibility that $\zero \in F$).

\begin{defn} \label{12} A (supertropical) \textbf{bilinear form} on supertropical vector spaces $V$
and $V' $ is a function $B : V\times V' \to \Fz$ that is a linear
functional in each variable.
\end{defn}

We write $\bil v{v'} $ for $B(v,v').$ Specifically, given $w \in
V',$ we can define the functional $\tlw : V \to F$ by $\tlw (v) =
\bil vw.$ (Similarly we define $\tlv :V \to F$ for $v \in V.$)

\begin{example} There is a natural bilinear form $B:V\times V^* \to
\Fz$, given by $\bil v f = f(v)$, for $v\in V$ and $f\in
V^*$.\end{example}

 \begin{rem}\label{natmap} $ $
\begin{enumerate} \eroman
    \item Notation as in Definition~\ref{12}, any bilinear form  induces a natural map $\Phi: V' \to  V^*, $ given by $w\mapsto
\tlw$. Likewise, there is a natural map $\Phi: V \to{V'}^*, $
given by $v\mapsto \tlv$.

 \pSkip

    \item  For any bilinear form $B$, if $v \lmod \sum_i \a _i v_i$ and $w
    \lmod
\sum_j \bt _j w_j$, for $\a_i, \bt_j \in F    ,$ then
\begin{equation}\label{0.4} \bil vw \lmodg  \sum _{i,j }  \alpha_i \bt_j\bil {v_i}{w_j} .\end{equation}
\end{enumerate}
\end{rem}

\begin{defn}
 When $V' = V$, we say that
$B$ is a \textbf{(supertropical) bilinear form} on the vector
space~$V.$  The space~$V$  is \textbf{nondegenerate} (with respect
to $B$)  if $\bil {v}V\not \subseteq \tG, \ \forall v \in V.$
\end{defn}

Although this definition suffices to carry through much of the
theory, we might want to compute the bilinear form $B$ in terms of
its values on an s-base of $V$. To permit this, we tighten the
definition a bit.

\begin{defn}
We say that a bilinear form $B$ is \textbf{strict} if
$$
\bil {\a_1v_1 + \a_2 v_2}{\bt_1w_1 + \bt_2 w_2}   = \a_1\bt_1 \bil
{v_1}{w_1} + \a_1\bt_2 \bil {v_1}{w_2} +    \a_2\bt_1 \bil
{v_2}{w_1} +\a_2\bt_2 \bil {v_2}{w_2} ,  $$ for $v_i \in V$ and
$w_i \in V'$.
\end{defn}

\begin{defn} The \textbf{Gram matrix} of the bilinear form with respect to vectors $v_1, \dots,
v_k \in V = \Fz^{(n)}$ is defined as the $k \times k$  matrix
\begin{equation}\label{eq:GramMatrix}
\tilG(v_1, \dots, v_k ) = \left( \begin{array}{cccc}
                      \bil {v_1}{v_1} &  \bil {v_1}{v_2} & \cdots & \bil {v_1}{v_k} \\
                      \bil {v_2}{v_1} &  \bil {v_2}{v_2} & \cdots & \bil {v_2}{v_k} \\
                         \vdots & \vdots &  \ddots & \vdots \\
                      \bil {v_k}{v_1} &  \bil {v_k}{v_2} & \cdots & \bil {v_k}{v_k} \\
                        \end{array} \right).
\end{equation}
\end{defn}


%
%
%
%
%
%
%
%
%

\begin{defn}  We write
 $v \gperp w$
when $\bil{v}{w} \in \tGz$,  that is $\langle v,w\rangle\lmodg
\fzero$. In this case, we say that $v$  is \textbf{left ghost
orthogonal} to $w$, or \textbf{left g-orthogonal} for short. Likewise,
a subspace $W_1$ is  \textbf{left g-orthogonal} to $W_2$ if
$\langle w_1,w_2\rangle\in \tGz$ for all $w_i\in W_i.$



A subset $S$ of $V$ is \textbf{g-orthogonal} (with respect to a
given bilinear form) if any pair of distinct vectors from $S$ is
g-orthogonal.
\end{defn}

In this paper we usually require $\gperp $  to be a symmetric
relation. This was studied in greater detail in~
\cite[Definition~6.12]{IzhakianKnebuschRowen2010LinearAlg}, but we
take the simpler definition here since we focus on strict bilinear
forms, for which the two notions coincide in view of
\cite[Lemma~6.15]{IzhakianKnebuschRowen2010LinearAlg}).

\subsection{Isotropic vectors}

\begin{defn} A vector $v \in V$ is \textbf{g-isotropic}  if $\bil v v \in \tGz$; $v$ is  \textbf{g-nonisotropic}   if $\bil v v \in \tT.$
A subset $ S\subset V$ is \textbf\textbf{g-nonisotropic}   if each
vector of $S$ is g-isotropic. \end{defn}

%

For any supertropical \semifield0 $F$ and  $k \in \Net$, we have
the sub-\semifield0
$$F^k = \{ a^k : a \in F\}.$$ For
example, when $F$ is the supertropical \semifield0 built from  the
ordered group $(\R,\cdot)$, then $F^2 \ne F,$ since we only get
the positive elements. However, when $F$ is the supertropical
\semifield0 built from the ordered group $(\R,+)$, or from $(\R
^+,\cdot)$, then $F^2 = F.$ 
\begin{defn}
A vector $v \in V$ is called   \textbf{normal} if $\bil vv =\fone.
$
\end{defn}

\begin{rem}\label{normal1}  Suppose $F^2 = F.$
If  $\bil vv  =  a \in \tT$, then $\bil {\frac v{\sqrt a}}{\frac
v{\sqrt a}}= \fone,$ so $\frac {v}{\sqrt a}$ is normal. Thus, in
this case, any g-nonisotropic vector has a scalar multiple that is
normal.
\end{rem}


The bilinear form $B$ is \textbf{supertropically alternate} if
each vector is g-isotropic; i.e., $\bil vv \in \tGz$ for all $v\in
V.$
 $B$ is \textbf{supertropically symmetric}  if $\bil vw + \bil
wv\in \tGz$ for all $v,w \in V.$ A special case: $B$ is
\textbf{symmetric} if $\bil vw = \bil wv$ for all $v,w \in V.$

\begin{lem}\label{18111} If $B$ is  supertropically symmetric on the vector
space $Fv_1+Fv_2$ and    $v_1$ and $v_2$ are both g-isotropic,
then $B$ is supertropically alternate.\end{lem}  \begin{proof}
 $\bil {\gm _1 v_1 + \gm _2 v_2}{\gm _1  v_1 + \gm _2
v_2} = \gm _1^2 \bil{ v_1}{ v_1} + \gm _1 \gm _2 (\bil{ v_1}{
v_2}+ \bil{ v_2}{ v_1} ) + \gm _2 ^2 \bil{ v_2}{ v_2}\in \tGz.$
\end{proof}
\begin{prop}\label{18112} If $B$ is  supertropically symmetric on a vector space with an s-base of g-isotropic
vectors,  then $B$ is supertropically alternate.\end{prop}
\begin{proof} Apply induction to the lemma.\end{proof}

We recall another way of verifying tropical dependence, in terms
of bilinear forms.

\begin{thm}[{\cite[Theorem~6.7]{IzhakianKnebuschRowen2010LinearAlg}}]\label{thm:GramMat}
If the vectors
 $w_1 \dots, w_k \in V$  span  a nondegenerate subspace $W$ of $V$ with
 $ | \tilG(w_1 \dots, w_k ) | \in \tGz$, then $w_1, \dots, w_k $
are tropically dependent.
\end{thm}

\subsection{The radical with respect to a bilinear form}

\begin{defn}
The \textbf{(left) orthogonal ghost complement} of $S\subseteq V$ is defined
as
$$S^{\perp} := \{ v\in V: \bil vS \in \tGz \}.$$
 The \textbf{radical}, $\rad( V)$, with respect to a given bilinear form $B$,
  is defined as $V^{\perp}.$ Vectors $w_i$ are \textbf{radically
dependent} if $\sum _i \a _i w_i \in \rad (V)$ for suitable $\a_i
\in \tTz,$ not all $\fzero.$

\end{defn}

Clearly, $\tHz \subseteq \rad(V)$.

\begin{rem} \label{15} $ $
\begin{enumerate}\eroman
    \item $\rad (V) = \tHz$ when
$V$ is nondegenerate, in which case radical dependence is the same
as tropical dependence. \pSkip
    \item

Any ghost complement $V'$ of $\rad (V)$ is obviously left
g-orthogonal to $\rad (V),$ and nondegenerate since $$\rad (V') \
\subseteq \ V'  \cap  \rad (V) \ \subseteq \tHz.$$  This
observation enables us to reduce many proofs to nondegenerate
subspaces, especially when the  Gram-Schmidt procedure  is
applicable (described below in Remark~\ref{16}).

\end{enumerate}
\end{rem}

 \begin{prop}\label{thm:CS}
   If
\begin{equation}\label{eq:CS}
 \bil v w  \bil  w v \lmodg \bil vv\, \bil ww,
\end{equation}
and the vector space $\Fz v+\Fz w$ is nondegenerate, then $v,w$
are tropically
 dependent on $\rad(\Fz v+\Fz w)$. Conversely, if $v,w$ are tropically
 dependent on $\rad(\Fz v+\Fz w)$ and $\bil v v$ and $\bil w w$ are tangible, then
 \eqref{eq:CS} holds.

 \end{prop}
\begin{proof}
   If $ \bil v w \, \bil w v \lmodg    \bil vv\, \bil ww$, then
$$\bigg| \vmMat{\bil vv}{\bil vw} {\bil wv} {\bil ww} \bigg| \in
\tGz,$$ so the vectors $v$ and $w$ are tropically
 dependent by Theorem \ref{thm:GramMat}. Conversely, if $\bil vv, \bil ww \in \tT$
 and $\bigg|  \vmMat{\bil vv}{\bil vw} {\bil wv} {\bil ww} \bigg| \in \tGz$,
 then necessarily
  $ \bil v w \, \bil  wv\lmodg \bil vv\, \bil ww$.
\end{proof}

\section{Cauchy-Schwartz spaces}

{\it For convenience, we assume throughout this section that that
$B$ is supertropically symmetric}, i.e., $\bil vw + \bil wv\in
\tGz$ for all $v,w \in V.$ This assumption is justified by the
following
 result from \cite{IzhakianKnebuschRowen2010LinearAlg}:

\begin{thm}[{\cite[Theorem
~6.19]{IzhakianKnebuschRowen2010LinearAlg}}]\label{orthogsym} If
g-orthogonality is a symmetric relation for the supertropical
bilinear form~$B$, then $B$ is  supertropically symmetric.
\end{thm}

Then we have:

 \begin{equation}\label{CS00}
\bil {v+w}{v+w} \lmodg \bil vv + \bil ww + (\bil vw + \bil wv)
\lmodg \bil vv + \bil ww . \end{equation}

%

\subsection{Compatible vectors}

Since we cannot subtract vectors, we  introduce a notion that
plays a key role in the supertropical theory.

\begin{rem} If
$$\bil vw + \bil wv \ge_\nu \bil vv + \bil ww,$$ then $v+w$ is
g-isotropic.  This is clear from the first  $\lmodg $ relation in
\eqref{CS00}.\end{rem}

To avoid such g-isotropic vectors, we formulate the following
definition:

\begin{defn}\label{compat} Vectors $v$ and $w$ are \textbf{\compatible} if
\begin{equation}\label{CS0}  \bil vv + \bil ww \ge_\nu
\bil vw +\bil wv;\end{equation}  \compatible\ vectors $v$ and $w$  are called \textbf{
\scompatible} if either $\bil vv  \nucong \bil ww$ or we have
 strict $\nu$-inequality   in \eqref{CS0}.

 \end{defn}

\begin{example}\label{quad0} If $v+w$ is g-nonisotropic, then  $v$ and $w$ are \scompatible. Indeed, $$\bil {v+w}{v+w} \lmodg \bil vv +
\bil ww + (\bil vw   + \bil wv)$$ is presumed tangible. But $\bil
vw + \bil wv$ is ghost, which means $\bil vw   + \bil wv <_\nu
\bil vv + \bil ww .$ \end{example}

\begin{lem}\label{quad0}    Compatible vectors satisfy $ \bil {v+w}{v+w} \lmodg \bil vv   + \bil ww  ,$
equality holding when $B$ is strict.\end{lem}
 \begin{proof}
 $ \bil {v+w}{v+w} \lmodg \bil vv  + \bil
ww+\bil vw   + \bil wv  \lmodg \bil vv   + \bil ww.$
To prove equality   when  $B$ is strict, note that this is
clear unless $\bil vv \nucong \bil ww,$ in which case both sides
are $\bil vv^\nu$.
\end{proof}

\subsection{Cauchy-Schwartz spaces}

 We are ready for the main kind of vector space.

\begin{defn}\label{CauS} A subset $S \subseteq V$ is  \textbf{weakly Cauchy-Schwartz}  if every pair of elements of $S$
satisfies the condition \begin{equation}\label{CS}\bil vv \bil ww
\ge_\nu \bil vw ^2 + \bil wv ^2.\end{equation} $S  $ is
\textbf{Cauchy-Schwartz}  if strict $\nu$-inequality holds in
Equation~\eqref{CS}. A space $ V$ is \textbf{Cauchy-Schwartz} if
it has a Cauchy-Schwartz s,d-base.
\end{defn}

For example, the standard base of $F^{(n)}$ (cf.~Example~\ref{sb})
with respect to the scalar bilinear form is Cauchy-Schwartz.

\begin{lem}\label{comp1} If $\{ v, w \}$ is
Cauchy-Schwartz (resp.~ weakly Cauchy-Schwartz), then  $v$ and $w$
are \scompatible\ (resp. ~\compatible).
 \end{lem}
\begin{proof}
 Clearly $\bil vv \bil ww \nule (\bil vv   +
\bil ww) ^2.$ Hence,
$$\bil vw ^2 \nule  (\bil vv + \bil
ww)^2,$$  implying $\bil vw \nule \bil vv + \bil ww.$ Analogously,
$\bil wv\nule \bil vv + \bil ww.$ The same argument works for
$\nu$-inequality.
\end{proof}
%

\subsection{Bilinear forms on a space of rank 2}

Much of the theory reduces to the rank 2 situation. We fix some
notation for this subsection. We say that $\{ v_1, v_2\}$ is a
\textbf{corner singular pair} if the Gram matrix of the bilinear
form $B$ with respect to $v_1$ and $v_2$ is $\nu$-matched to a
matrix of the form $\vmMat{\a}{\a \bt}{\a \bt}{  \a \bt ^2}.$

Given vectors $v_1,v_2$ in $V$, let $\a_{ij} = \bil {v_i}{v_j},$
and put $\a = \a_{12} + \a_{21}\in \tGz.$ By symmetry we may
assume that $\a_{11} \le_{\nu} \a_{22}.$ Take $\hal \in \tT$ such
that $\hal  \nucong \a,$ and $\hal_{22} \in \tT$ such that
$\hal_{22} \nucong \a_{22} .$

 In contrast to the classical situation, any space $V$ of rank $\ge 2$ must have g-isotropic vectors,
as seen in the following computation.
\begin{example}\label{1811}

Given $\bt \in \tT,$ we define $w = v_1+\bt v_2$ and have
\begin{equation*}
\begin{array}{lll}
 \bil ww & = & \a_{11} +\a \bt + \a_{22} \bt^2; \\[1mm]
\bil w{v_2} & = &  \a_{12} + \a_{22}\bt;   \\[1mm] \bil {v_2}w & = &
\a_{21} + \a_{22}\bt. \end{array}   \end{equation*}

\begin{description}
    \item[CASE I] $\a_{22} = \fzero.$ Then $\bil ww = \a \bt$, $\bil w{v_2}
=  \a_{12} ,$ and  $\bil {v_2}w = \a_{21}$. Thus, $v_2$ and $w$
are \compatible, and we have reduced to the next case. \pSkip

\item[CASE II] $\a_{22} \ne \fzero.$

 Take $\bt $ large enough; i.e., $\bt  >_\nu
\frac{\a}{\hal_{22}} + 1+ \frac{\a_{11}}{\hal}$. (We discard the
last summand when  $\a = \fzero$.) Then $$\bil ww = \a_{22} \bt^2;
\qquad \bil w{v_2} =   \a_{22}\bt; \qquad    \bil {v_2}w =
 \a_{22}\bt,$$ so $\{w, v_2\}$ is a corner singular pair.
 In particular, when $v_2$ is g-nonisotropic, replacing $v_1$ by $w$ for large
 $\bt$ gives us an independent  pair of g- nonisotropic vectors, but
 at the cost of corner singularity. Next, we look for g-isotropic
 vectors. \pSkip

\item[CASE II.a]   $ \a ^2 \nug  \a_{11}\a_{22}.$
%
At any rate, taking $\bt =  \frac  {\a }{\hal_{22}}$ yields
\begin{equation*} \bil ww =  { \a \bt }^\nu = { \frac  {\a ^2
}{\hal_{22}} }^\nu; \qquad \bil w{v_2} = \a = \bil {v_2}w
\end{equation*}
Thus, the Gram matrix of the bilinear form $B$ with respect to
$v_2$ and $w$ is $\vmMat{\a_{22} }{ \a  }{\a } { \frac {\a ^2
}{\hal_{22}}^\nu},$   so again we have corner singularity.

                        On the other hand,  taking $\bt =  \frac  {\a_{11}
}{\hal}$ yields \begin{equation*}\bil ww = \a_{11}^\nu ; \qquad
\bil w{v_2} =  \a_{12} + \frac  {\a_{11} \a_{22}}{\hal}  ; \qquad
\bil {v_2}w = \a_{21} +   \frac  {\a_{11} \a_{22}}{\hal}.
\end{equation*}
Thus, the Gram matrix of the bilinear form $B$ with respect to
$v_2$ and $w$ is $\vmMat{\a_{22}}{ \gm  }{\dl }{\a_{11}^\nu},$
with $\gm  + \dl  \nucong \a.$

The pair $\{ v, w \}$  is not corner singular, since  $\a_{11}\a_{22} <_
\nu  \a^2.$

For the situation  $\a \in \tGz$,  we see that $w$ is g-isotropic
when $\a \bt \ge_ \nu \a_{11}$ and $\a \bt \ge_ \nu \a_{22}
\bt^2$, i.e.,
$$ \frac  {\a }{\hal_{22}}  \ge_\nu
\bt \ge_\nu \frac{\a_{11}} {\hal }.$$ We call this range of
$\bt$ the \textbf{g-isotropic strip} of the plane. 
\pSkip

\item[CASE II.b]  $ \a ^2 \le_\nu  \a_{11}\a_{22}$. 
Then $\bil ww = \a_{11}  + \a_{22}
                        \bt^2.$ Thus, $w$ is g-isotropic for $\bt^2 \nucong \frac {
                        \a_{11}}{\hal_{22}}.$
\end{description}

\end{example}

\begin{lem}\label{degen} When $F^2 = F,$
any space of tangible rank at least two contains an g-isotropic
vector.\end{lem} \begin{proof} By Example~\ref{1811}, since we
have $\a \in \tGz$.
\end{proof}

%


\subsection{The Gram-Schmidt procedure}

We start with a standard sort of calculation.

\begin{rem} \label{16}  (\textbf{The Gram-Schmidt procedure}) Suppose $W\subset V$ is   supertropically  spanned by an g-orthogonal set
$\tB = \{b_1, \dots, b_m\}$ for which each $\bil {b_j}{b_j}\ne
\fzero $. We take $\bt_j \in \tT$ for which $\bt_j \nucong \bil
{b_j}{b_j}.$  Then for each $j = 1 , \dots,  m$, and for any $v\in V,$
the vector $$v_\tB \ds = \sum _{j=1}^m \frac {\bil v
{b_j}}{\bt_j}b_j \in W
$$ satisfies
\begin{equation}\label{GSch0} \bil{v }{v_\tB}  \lmodg   \sum
_j \frac {\bil v {b_j}^2}{\bt_j}  \quad \text{and}  \quad \bil
{v_\tB}v \lmodg \sum _j \frac {\bil v {b_j} \bil {b_j} v }{\bt_j}
.\end{equation} The vector $v'_\tB =  v + v_\tB$ satisfies
\begin{equation}\label{GSch00}\bil {v'_\tB}{ b_i}
\lmodg \bil v{b_i}^\nu + \sum_{j\ne i} \frac {\bil v
{b_j}^2}{\bt_j}\bil{b_j}{b_i} \in \tGz.\end{equation} Hence,
$v'_\tB \gperp b_i$ for each $i$,
implying $v'_\tB \gperp W$. 
Furthermore,
\begin{equation}\label{GSch} \bil {v'_\tB}{v'_\tB} \lmodg \bil vv +  \sum
_j \frac {\bil v {b_j}(\bil v {b_j}+\bil   {b_j}v )}{\bt_j} +
\sum_{j , k} \frac {\bil v {b_j}\bil v {b_k} \bil {b_j} {b_k}}
{\bt_j\; \bt_k}.\end{equation} Equality holds in Equations
~\eqref{GSch0}, \eqref{GSch00}, and ~\eqref{GSch} when the
bilinear form $B$ is strict.
\end{rem}

\begin{lem}\label{GS0} Notation as in Remark \ref{16}, suppose that $\tB$
is weakly Cauchy-Schwartz. Then
$$\bil {v'_\tB}{v'_\tB} \lmodg \bil {v }{v }  +  \sum
_j \frac {\bil v {b_j}(\bil v {b_j}+\bil   {b_j}v )}{\bt_j} ,$$
equality holding when the bilinear form $B$ is strict.
\end{lem}
\begin{proof} The first assertion is clear unless $\bil {v'_\tB}{v'_\tB}$ is tangible, which means that there
is one term in the right of \eqref{GSch} which dominates all
others, and again we have the first assertion unless this term
comes from $\sum_{j , k} \frac {\bil v {b_j}\bil v {b_k} \bil
{b_j} {b_k}} {\bt_j\; \bt_k}.$

Now note that ${v'_\tB}$ is g-orthogonal to each $b_i$, as
observed above.

Also, by hypothesis, \begin{equation}\label{use1} \bt_j  \bt_k
\nuge  \bil {b_j}{b_k}^2,\end{equation} so multiplying both sides
by $\frac {\bil v {b_j}^2 \bil v {b_k}^2}{\bt_j^2\bt_k^2} $ yields
$$\frac {\bil v {b_j}^2}{\bt_j}\frac {\bil v {b_k}^2}{\bt_k} \ds \nuge \frac {\bil v {b_j}^2 \bil v {b_k}^2
\bil{b_j}{b_k}^2}{\bt_j^2 \bt_k^2};$$ it follows that either
$$\frac {\bil v {b_j}^2}{\bt_j} \nuge \frac {\bil v {b_j} \bil v
{b_k} \bil{b_j}{b_k}}{\bt_j \bt_k}\qquad \text{or} \qquad \frac
{\bil v {b_k}^2}{\bt_k} \nuge \frac {\bil v {b_j} \bil v {b_k}
\bil{b_j}{b_k}}{\bt_j\bt_k}.$$ Thus the  terms in the right side
of \eqref{GSch00} are all weakly dominated by the $\frac {\bil v
{b_j}^2}{\bt_j},$ and thus get absorbed by the terms $\frac {\bil
v {b_j}(\bil v {b_j}+\bil   {b_j}v )}{\bt_j},$ since $\bil v
{b_j}+\bil   {b_j}v \in \tGz$ by hypothesis.

%
%
%
Hence,
\begin{equation}\label{GSch1} \bil {v'_\tB}{v'_\tB} \lmodg \bil vv  +  \sum
_j \frac {\bil v {b_j}(\bil v {b_j}+\bil   {b_j}v
)}{\bt_j}.\end{equation} The last assertion is now clear,  and
equality holds at each stage of our argument when $B$ is strict
(since again the terms in the right side of \eqref{GSch00}
disappear).
\end{proof}

Looking carefully at the proof, we have the \textbf{dominant
index} (or indices) $j'$ such that $$\frac { (\bil v {b_{j'}}
+\bil {b_{j'}}v)^2}{\bt_{j'}} \ds \ge\! _\nu \  \sum _j \frac {
(\bil v {b_j}+\bil {b_j}v)^2}{\bt_j} ,$$ and $v'_\tB$ is
g-nonisotropic if $v$ and $b_{j'}$ are Cauchy-Schwartz for all
dominant indices. Consequently, we have the following result.

\begin{prop} The vector $v'_\tB$ is g-nonisotropic if $v$ is g-nonisotropic and
$\bil vv \nug \sum _j \frac { (\bil v {b_j}+\bil {b_j}v
)^2}{\bt_j}$, which is true when $v$ and $b_j$ are Cauchy-Schwartz
for each $j$. In this case,  $\bil {v'_\tB}{v'_\tB} = \bil vv.$

Conversely,    for a dominant index $j$, the ghost $\bil v
{b_j}+\bil {b_j}v$ dominates $\bil vv$ when the vectors $v$ and~
$b_j$ are not Cauchy-Schwartz, implying  $v'_\tB$ is g-isotropic.
\end{prop}
\begin{proof}  $\bil vv$ dominates $
\sum _j \frac { (\bil v {b_j}+\bil {b_j}v )^2}{\bt_j}$ if $v$ and
$b_j$ are Cauchy-Schwartz for each $j$.

Conversely,  $v'_\tB$ is g-isotropic when the vectors $v$ and
$b_j$ are not Cauchy-Schwartz.
\end{proof}

\begin{defn}
A  space is \textbf{anisotropic} if it  has a g-orthogonal
d,s-base of g-nonisotropic vectors which is Cauchy-Schwartz.
\end{defn}

Our definition of anisotropic space is weaker than the classical
definition, as it must be in view of Lemma~\ref{degen}.
The situation becomes clearer when the bilinear form $B$ is strict.  Applying
induction to Lemma~\ref{GS0}, we have the following conclusion:
\begin{prop}\label{CS1}
 When the bilinear form $B$ is strict, any nondegenerate space supertropically spanned by a
Cauchy-Schwartz set $S = \{ s_1, \dots, s_n \}$  is itself
Cauchy-Schwartz. \end{prop}
\begin{proof}
\pSkip (i) Letting  $v= \sum \a_i s_i$ and $w = \sum \bt_j s_j$,
then $\bil vw = \sum_{i,j} \al_{i}\bt_{j}\bil{s_i}{s_j}$. By the
Frobenius formula~\eqref{eq:Frobenius},
\begin{equation*}\label{eq:2} \bigg(\sum_{i,j} \al_{i}\bt_{j}\bil{s_i}{s_j}\bigg)^2 \ds =
\sum_{i,j} \al_{i}^2\bt_{j}^2\bil{s_i}{s_j}^2,
\end{equation*}
which is dominated by $\sum_{i,j}
\al_{i}^2\bt_{j}^2\bil{s_i}{s_i}\bil{s_j}{s_j}$ since $s_i$ and
$s_j$ are Cauchy-Schwartz, which in turn is dominated by $ \bil
vv\, \bil ww.$ \end{proof}

\begin{defn} A space $W$ is the \textbf{g-orthogonal sum} of two
subspaces $W_1$ and $W_2$ if $W_1+W_2 = W$ and $\bil {W_1}{W_2}
\subseteq \tGz,$ such that any pair of vectors $\{ w_1, w_2\}$
with $w_i \in W_i$ is not Cauchy-Schwartz.\end{defn}


\begin{thm}\label{decompthm}
Any space V with a (supertropically symmetric) bilinear form $B$
has a thick subspace which is the g-orthogonal sum of an
anisotropic subspace  and a supertropically alternate subspace.
\end{thm}
\begin{proof} Apply the Gram-Schmidt procedure as far as possible to obtain the anisotropic subspace
$W$ with g-nonisotropic Cauchy-Schwartz d-base $\tB_W = \{ w_1,
\dots, w_m \}.$ If we have some other g-nonisotropic vector $v \in
V\setminus W$ such that $\{ w_j, v \}$ is  Cauchy-Schwartz for
some $j$, then, in view of Example~\ref{1811}, we can replace $v$
by $\bt w_j + v$ for large enough $\bt$, which is Cauchy-Schwartz
with each of $w_1, \dots, w_m$, and thus expand $W$ by another
application of the Gram-Schmidt procedure, a contradiction.

We conclude that no vector in $V\setminus W$ is  Cauchy-Schwartz
with any member of $\tB _W$. We expand $\tB_W$ to a d-base $\tB$
of $V$. The remaining vectors in $\tB \setminus \tB_W$ are not
Cauchy-Schwartz with $W$ and thus produce g-isotropic vectors, in
view of~Lemma~\ref{GS0}, and the space they generate is
supertropically alternate in view of~Proposition~\ref{18112} .
\end{proof}

\section{Supertropical quadratic forms}

Let us elaborate on the Cauchy-Schwartz property to get a
supertropical version of quadratic forms. As in the classical
case, given  a supertropical  bilinear form~$B$, we define $Q_B(v)
:= \bil vv.$ The following observation is easy but surprising.

\begin{prop}\label{q0} If $V$ is Cauchy-Schwartz with respect to   a supertropically symmetric  bilinear
form~$B$, then $Q_B(v+w)  \lmodg Q_B(v) + Q_B(w)$ for all $v,w \in
V$. Furthermore, if the bilinear form $B$ is strict, then
$Q_B(v+w) = Q_B(v) + Q_B(w)$.
\end{prop}
\begin{proof} $Q_B(v+w) = \bil {v+w}{v+w}   \lmodg \bil vv + \bil ww =Q_B(v) +
Q_B(w)$, in view of \eqref{CS0} and Lemma~\ref{comp1}.
 When  $B$
is strict, we get the second assertion by Lemmas~\ref{quad0} and ~\ref{comp1}.
 \end{proof}


\begin{defn}\label{quasil} A \textbf{(supertropical) quasilinear quadratic form} on a vector space $V$ is a function
$Q  : V \to F$ satisfying
$$ Q(\a v) = \a^2 Q(v), \quad  Q( v+w) \lmodg  Q(v)+Q(w), \quad  \forall \a \in F, \ v,w \in V.$$
 The quadratic form $Q$ is  \textbf{\squasilinear} if it satisfies
\begin{equation}\label{quad2} Q(v+w) = Q(v)+Q(w), \quad \forall   v,w \in
V. \end{equation} (In this case, we also say that the quadratic
space $V$ is \textbf{\squasilinear}.)

A \textbf{(supertropical)  quasilinear quadratic space} is a pair
$(V,Q)$ where $Q : V\to F$ is a  quasilinear quadratic form. We
say that the
  space is \textbf{\squasilinear} when the underlying
quadratic form is \squasilinear. \end{defn}

 Since supertropical algebra has ``characteristic 1'' and, in particular,  often has properties of
characteristic~2, one should expect the  theory of supertropical
quadratic forms also to behave as the classical theory of
quasilinear quadratic forms over fields of characteristic 2, which
is treated for example in \cite[\S1.6]{K} and \cite[II,\S10]{EKM}
(where the term ``totally singular''  is used instead of
``quasilinear'').

\begin{rem}\label{rise} By Proposition~\ref{q0}, any    supertropically symmetric bilinear form $B$
gives rise to a  quasilinear quadratic form~$ Q(v): = \bil vv,$
which is strictly quasilinear when $B$ is strict.
\end{rem}

\begin{defn}
We call $Q$ of Remark~\ref{rise} the \textbf{quadratic form
associated to}  $B$, and say that $Q$ \textbf{admits} the bilinear
form $B$.
\end{defn}

 In particular, $v$ is g-isotropic iff $Q(v)\in
\tGz.$ Note that a quadratic form $Q$ may admit many different bilinear forms.

The quadratic form associated to a non-symmetric bilinear form
  might fail to be quasilinear.
\begin{example}  Take the quadratic form
associated to the   bilinear form $B$ with base $\{e_1, e_2\}$
whose matrix is $\vmMat{\fzero}{\fone}{\fzero }{ \fzero},$ and $b
= e_1 + e_2$. Then $Q(b) = \bil {e_1}{e_2}   = \fone . $ This is
not quasilinear.
\end{example}

In this paper we have been focusing on symmetric bilinear forms,
and thus we treat quasilinear quadratic forms. Although we focus
on \squasilinear \ quadratic forms here, here is an important
example that is not \squasilinear.

\begin{example}[The hyperbolic plane]\label{hp} $ $
\label{hyplane}  \begin{enumerate} \eroman

    \item  Take the quadratic form
associated to the   bilinear form $B$ whose matrix  is
$\vmMat{\fzero}{\fone}{\fone  }{ \fzero}.$ Thus, $e_1$ and $e_2$
are isotropic. Take $b = e_1 + e_2$. Then $Q(b) = \bil
{e_1}{e_2}^\nu = \fone^\nu.$ \pSkip

\item More generally, we say that $V$ is a \textbf{supertropical
hyperbolic plane} if it has a base $\{e_1, e_2\}$ of g-isotropic
vectors, for which $Q(e_1 +e_2) >_{\nu} Q(e_1) + Q(e_2).$ \pSkip

\item Any supertropical hyperbolic plane has
  a symmetric bilinear form with respect to which $\bil {e_1}{e_2} \in \tT$ and $e_1, e_2$
  are not Cauchy-Schwartz; namely we define $\bil {e_1}{e_2} = \bil {e_2}{e_1}   \in \tT$ to be
  $\a \in \tT$ for which $\a  \nucong
Q(e_1 +e_2) .$  \end{enumerate}

\end{example}

\begin{defn}   The \textbf{orthogonal sum}  $Q=Q_1+Q_2$ of two
quadratic spaces $(V_1,Q_1)$ and $(V_2,Q_2)$  is defined as $(V_1
\oplus V_2, Q)$ where $Q(v_1,v_2) = Q_1(v_1)+Q_2(v_2).$

 An orthogonal sum of hyperbolic
planes is called a \textbf{hyperbolic space}.
 \end{defn}

\begin{lem} Any \squasilinear \ quadratic space   has a thick subspace which
is an orthogonal sum of 1-dimensional quadratic subspaces.
\end{lem}\begin{proof} Just take a d-base.\end{proof}

\begin{example} In view of Proposition~\ref{q0}, the quasilinear quadratic form obtained
from any Cauchy-Schwartz base with respect to a supertropically
symmetric, strict  bilinear form $B$ is  \squasilinear.\end{example}

\begin{rem}\label{bf}
 Conversely, given
a \squasilinear \ quadratic form $Q$   over a  semifield $F$
satisfying $F = F^2$, we have a canonical  bilinear form $B_Q$
admitted by $Q$, given by:

\begin{equation} B_Q(v,w) = \sqrt{Q(v) Q(w)}.\end{equation}
\end{rem}

\begin{thm}\label{quadlin} If $(V,Q)$ is a \squasilinear \ quadratic space, then $B_Q$ is  a strict,   symmetric bilinear form,
with respect to which $V$ is Cauchy-Schwartz.\end{thm}

\begin{proof} $\bil {v+v'}w ^2 = Q(v+v')Q(w) = Q(v)Q(w) +
Q(v')Q(w) = \bil v w ^2 +  \bil {v'} w ^2  = (\bil v w   +  \bil
{v'} w)^2. $ Taking square roots shows that $B_Q$ is a strict
bilinear form, which is obviously symmetric.\end{proof}

\begin{rem} In view of Theorem~\ref{decompthm}, the quasilinear quadratic form
of a symmetric bilinear form can be decomposed into the sum of a
\squasilinear \ quadratic form and a hyperbolic space.\end{rem}
%
%


\end{document}

%% file: tropMacro.tex
\usepackage{epsfig,latexsym,amsfonts,amssymb,amsmath,amscd,graphics,epic}
\usepackage{amsfonts,amssymb,amsmath,amscd,amsthm}
\usepackage[mathscr]{eucal}
\usepackage{mathrsfs}
\usepackage{oldgerm,units}
\usepackage{wrapfig,epsfig}

\usepackage{ifthen}
\usepackage{mathbbol}

\usepackage{amsthm}
\usepackage[mathscr]{eucal}
\usepackage{mathrsfs}
\usepackage{mathbbol}
\usepackage{oldgerm,units}
\usepackage{wrapfig}

\newtheorem{theorem}{Theorem}[section]

\newtheorem{definition}[theorem]{Definition}

\newtheorem{example}[theorem]{Example}

\newcommand{\Real}{\mathbb R}

\newcommand{\Net}{\mathbb N}

\newcommand{\one}{\mathbb{1}}
\newcommand{\zero}{\mathbb{0}}

\newcommand{\trop}[1]{\mathcal{#1}}

\newcommand{\tB}{\trop{B}}

\newcommand{\tG}{\trop{G}}

\newcommand{\tT}{\trop{T}}

\newcommand{\To}{\longrightarrow }

\newcommand{\ep}{\epsilon}
\newcommand{\al}{\alpha}
\newcommand{\bt}{\beta}
\newcommand{\gm}{\gamma}

\newcommand{\dl}{\delta}

\newcommand{\OP}{\left(}
\newcommand{\CP}{\right)}

    \ifx\proof\undefined
    \newenvironment{proof}{
    \smallskip
    \noindent\emph{Proof.}}{\hfill\(\Box\)
    \bigskip
    } \fi

\newcommand{\ifdef}[3]{\ifthenelse{\equal{#1}{true}}{#2}{#3}}